\documentclass[11pt]{amsart}

\usepackage{amsthm}
\usepackage{amsmath}
\usepackage{amssymb}

\usepackage{a4wide}
\usepackage{bbm}
\usepackage{graphicx}

\newcommand{\RR}{\mathbbm{R}}
\newcommand{\ZZ}{\mathbbm{Z}}
\DeclareMathOperator{\interior}{int}
\DeclareMathOperator{\conv}{conv}

\theoremstyle{plain} 
\newtheorem{theorem}{Theorem}

\newtheorem{proposition}[theorem]{Proposition}
\newtheorem{corollary}[theorem]{Corollary}

\theoremstyle{definition}  
\newtheorem{definition}{Definition}
\newtheorem{example}{Example}

\theoremstyle{remark}  
\newtheorem*{remark}{Remark}

\renewcommand{\P}{{\mathcal P}}
\newcommand{\F}{{\mathcal F}}
\newcommand{\G}{{\mathcal G}}
\newcommand{\e}{{\mathbf e}}

\renewcommand{\v}{{\mathbf v}}
\renewcommand{\t}{{\mathbf t}}
\newcommand{\x}{{\mathbf x}}
\renewcommand{\u}{{\mathbf u}}

\newcommand{\y}{{\mathbf y}}
\renewcommand{\a}{{\mathbf a}}

\newcommand{\p}{{\mathbf p}}
\newcommand{\0}{{\mathbf 0}}

\newcommand{\N}{{\mathbb N}}
\newcommand{\Z}{{\mathbb Z}}

\newcommand{\R}{{\mathbb R}}

\newcommand{\V}{{\mathcal V}}
\newcommand{\T}{{\mathcal T}}
\newcommand{\M}{{\mathcal M}}
\newcommand{\A}{{\mathcal A}}

\newcommand{\Lat}{{\mathcal L}}
\renewcommand{\L}{{\mathcal L}}
\renewcommand{\S}{{\mathcal S}}

\newcommand{\wordlength}{\operatorname{w}}
\newcommand{\growth}{\operatorname{S}}
\newcommand{\cone}{\operatorname{cone}}

\newcommand{\triang}{\operatorname{triang}}
\newcommand{\total}{\operatorname{total}}

\begin{document}

\title{Root polytopes and growth series of root lattices}


\author{Federico Ardila}
\email{\vspace{0.1cm}federico@math.sfsu.edu, beck@math.sfsu.edu, serkan@math.sfsu.edu}
\address{Department of Mathematics, San Francisco State University.\vspace{-0.5cm}}
\author{Matthias Beck}
\author{Serkan Ho\c{s}ten}
\author{Julian Pfeifle}
\email{julian.pfeifle@upc.edu}
\address{Departament de Matem\`atica Aplicada II, Universitat Polit\`ecnica de Catalunya.}
\author{Kim Seashore}
\email{kseashore@sfsu.edu}
\address{SESAME (Graduate group in Math and Science Education), Univ. of California, Berkeley.}

\thanks{Research partially supported by NSF grants  DMS-0801075
  (Ardila) and DMS-0810105 (Beck), and MEC grants MTM2005-08618-C02-01 and MTM2006-1267 (Pfeifle)}


\begin{abstract}
  The convex hull of the roots of a classical root lattice is called a
  \emph{root polytope}. We determine explicit unimodular
  triangulations of the boundaries of the root polytopes associated to
  the root lattices $A_n$, $C_n$ and $D_n$, and compute their $f$-and $h$-vectors. 
  This leads us to recover formulae for the growth series of these root lattices,
  which were first
  conjectured by Conway--Mallows--Sloane and Baake--Grimm and proved by
  Conway--Sloane and Bacher--de la Harpe--Venkov.
  \end{abstract}

\date{26 September 2008}

\maketitle


\section{Introduction}

A \emph{lattice} $\Lat$ is a discrete subgroup of $\R^n$ for some $n
\in \Z_{> 0}$. The \emph{rank} of a lattice is the dimension of the
subspace spanned by the lattice. We say that a lattice $\Lat$ is
\emph{generated as a monoid} by a finite collection of vectors $\M
= \{\a_1,\dots,\a_r\}$ if each $\u \in \Lat$ is a
nonnegative integer combination of the vectors in $\M$. For
convenience, we often write the vectors from $\M$ as columns of a
matrix $M \in \R^{n \times r}$, and to make the connection between
$\Lat$ and $M$ more transparent, we refer to the lattice generated by $M$ as $\Lat_M$.
The \emph{word length} of $\u$ with respect to $\M$, denoted
$\wordlength(\u)$, is $\min(\sum c_i)$ taken over all expressions $\u = \sum c_i \a_i$ with $c_i \in \Z_{\ge 0}$. The
\emph{growth function} $\growth(k)$ counts the number of elements
$\u\in \Lat$ with word length $\wordlength(\u) = k$ with respect to
$\M$.  We define the \emph{growth
series} to be the generating function $G(x):= \sum_{k \geq 0} \growth(k) \,
x^k$. It is a rational function $G(x)= \frac{h(x)}{(1-x)^d}$ where $h(x)$ is
a polynomial of degree less than or equal to the rank $d$ of $\Lat_M$ \cite{benson}.
We call $h(x)$ the \emph{coordinator polynomial} of the growth
series. It is important to keep in mind that these functions all
depend on the choice of generators for the monoid.

In this paper we examine the growth series for the classical root
lattices $A_n$, $C_{n}$ and $D_n$, generated as monoids by their standard set of generators. 
Conway and Sloane  \cite{conway} proved an explicit formula for the growth series for $A_n$ and, with Mallows's help, conjectured one for the $D_n$~case. Baake and Grimm \cite{baake} later conjectured formulae for the $B_n$ and $C_n$ cases. Bacher, de la Harpe, and Venkov \cite{bacherharpevenkov} subsequently provided the proofs of all these cases. We give alternative proofs of the formulae in the cases $A_n, C_n,$ and $D_n$, by computing the $f$-vector of a unimodular triangulation of the corresponding root polytope.

The approach presented here is a natural
extension of the proofs related to the growth series of cyclotomic
lattices given in \cite{cyclotomic}.
We let $\P_{\M}$ be the polytope formed by the convex
hull of the generating vectors in $\M$. For the lattices we consider, this polytope is the root polytope of the corresponding lattice. We determine
explicit  unimodular triangulations of these polytopes and show that the 
$h$-polynomial of these triangulations (hence, of any unimodular triangulation) is identical to
the coordinator polynomial $h(x)$ for the respective root lattice. 
Our method implies that $h(x)$ is necessarily  palindromic 
and must have nonnegative coefficients; this is confirmed by the formulae.
Incidentally, since the coordinator polynomial for the growth series of the root lattice $B_n$ is not palindromic, our approach would need to be modified to prove the formula in the $B_n$~case.

To state our main results, let $A_{n}$, $C_n$, and $D_n$ be the classical root lattices generated as monoids by 
\begin{eqnarray*}
\M_{A_{n}} &=& \{\e_i - \e_j : 0 \leq i,j \leq n+1\}
\\
\M_{C_n} &=& \{\pm 2\e_i : 1 \leq i \leq n \, ;\,  \pm \e_i \pm \e_j:
1  \leq i \neq j \leq n\},
\\
\M_{D_n} &=& \{\pm \e_i \pm \e_j: 1
\leq i \neq j \leq n\},
\end{eqnarray*}
respectively, and define the classical root polytopes $\P_{A_{n}} =
\conv \M_{A_n} $, $\P_{C_n} = \conv \M_{C_n}$, and $\P_{D_n} = \conv
\M_{D_n}$ to be the convex hulls of these generating sets.

The \emph{$f$-vector} of a simplicial complex $\Gamma$ is
given by $f(\Gamma) = (f_{-1},f_0,f_1,\dots,f_{n-1})$ where
$f_i$ is the number of $i$-dimensional
faces of $\Gamma$; by convention $f_{-1}=1$. The \emph{$f$ and $h$-polynomials} of $\Gamma$ are defined \cite{Ziegler}\footnote{Some authors use a slightly different definition.}
 to be
\[f_{\Gamma}(x)= \sum_{i=-1}^d f_{i} \, x^{d-i} \qquad h_{\Gamma}(x) = f_{\Gamma}(x-1)=  \sum_{i=-1}^d
f_{i} \, (x-1)^{d-i}.\]

\begin{theorem}\label{fpolys}
Let $f_{A_n}(x)$, $f_{C_n}(x)$, and $f_{D_n}(x)$ be the $f$-polynomials of any unimodular triangulations of the boundaries of the classical root polytopes. Then
\begin{eqnarray*}
  f_{A_n}(x) &=& \sum_{m=0}^n {n+m \choose m,m,n-m}x^{n-m}
  ,  \\
  f_{C_n}(x) &=& 
  \sum_{m=0}^n
  \frac{n \, 2^{2m}}{n+m}\binom{n+m}{2m} x^{n-m}
  , \\
    f_{D_n}(x) &=&  
  \sum_{k=0}^n
  \left(
  \frac{n \, 2^{2m}}{n+m}\binom{n+m}{2m} - 
    \frac{n(2n-m-1)2^{m-1}}{n-m}{n-2 \choose m-1}
  \right) x^{n-m}.
\end{eqnarray*}
\end{theorem}

\begin{theorem}\label{maintheorem}\cite{bacherharpevenkov, conway}
The
coordinator polynomial of the growth series of the classical root lattices $A_n, C_n,$ and $D_n$ with respect to the generating sets $\M_{A_n}, \M_{C_n}$, and $\M_{D_n}$
 is equal to the
$h$-polynomial of any  unimodular
triangulation of the respective polytopes $\P_{A_n}$, $\P_{C_n}$, and $\P_{D_n}$. These polynomials 
are palindromic and have nonnegative coefficients. More specifically, 
  \[ h_{A_n}(x) = \sum_{k=0}^{n} \binom{n}{k}^2 x^k,  \quad
  h_{C_n}(x) = \sum_{k=0}^n   \binom{2n}{2k} x^k, \quad
  h_{D_n}(x) =  \sum_{k=0}^n \left[\binom {2n}{2k} -\frac{2k(n-k)}{n-1} \binom n k \right] x^k.\]
 \end{theorem}



\section{Finding growth series from unimodular triangulations}

Our proof of Theorem \ref{fpolys} is combinatorial, and now we show how to deduce Theorem \ref{maintheorem} from it.  
First we need some definitions. The \emph{$h^*$-polynomial} of a $d$-dimensional lattice polytope $P$ in a lattice $\L$ is defined by 
\[
1 + \sum_{r > 0} |rP \cap \L|\,\,t^r = \frac{h^*_{P}(t)}{(1-t)^{d+1}}.
\]

\begin{theorem} \label{h*} 
\cite{betke, stanley80}
If $P$ is a $d$-dimensional lattice point configuration
and $\Gamma$ is a unimodular triangulation of $P$, then $h^*_{\conv(P)}(t) = t^d \, h_{\Gamma}(\frac1t)$.
\end{theorem}

For $\M=\M_{A_n}, \M_{C_n}, \M_{D_n}$, we will construct an explicit unimodular triangulation $\Gamma_{\M}$ of the boundary of the root polytope $P_{\M}$ (and of the point configuration $\M \cup 0$ by coning through the origin; this operation doesn't affect the $h$-vector). The $h$-polynomial of $\Gamma_\M$ will give us the $h^*$-polynomial of $P_\M$, which equals the coordinator polynomial of $\L_\M$ since $P_\M$ has a unimodular triangulation.

\medskip

The point of view of monoid algebras will also be useful. 
We start with the 
$(n\times r)$-matrix $M=(\a_1, \a_2,\dots,\a_r)$ whose columns generate $\Lat_M \subset \R^n$ 
as a monoid. We define $\P_M$ to be the convex hull of these generators, namely the polytope
\[
\P_{M} = \conv \{\a_1, \ldots, \a_r\} = \left\{\p \in \R^n : \p = \sum_{i=1}^r \lambda_i \a_i, \,
  \lambda_i \geq 0 \,\,\text{and} \,\,\sum_{i=1}^r \lambda_i = 1\right\}.
\] 
In the cases we study, 
the polytope $\P_M$ has $\a_1, \ldots, \a_r$ and the origin as its only lattice points. Moreover, the 
origin is the unique interior lattice point. Motivated by this  we let
\[M' = \left (
    \begin{array}{ccccc}
         1 & 1 & \dots & 1 & 1\\
     \a_1 & \a_2 & \dots & \a_r & 0\\
    \end{array}
  \right ).\] 
We also define  $\S(M') \subset \R^{n+1}$ as the monoid generated by the columns of $M'$. This monoid
is contained in the rational polyhedral cone 
\[ \cone(M') = \left\{ \left( \begin{array}{c} k \\ \u \end{array} \right) \in \R^{n+1} :
 \left( \begin{array}{c} k \\ \u \end{array} \right) = \sum_{i=1}^r
\lambda_i  \left( \begin{array}{c} 1 \\ \a_i \end{array} \right) + 
\lambda_{r+1}  \left( \begin{array}{c} 1 \\ \0 \end{array} \right)   \,\,\text{and} \,\, \lambda_1, \ldots, \lambda_{r+1}  \geq 0 \right\}. 
\]
Note that for any $k \in \R_{>0}$ we have $\left\{ \u \in \R^n : (k, \u) \in \cone(M')\right\} = k\P_M$.  In general, $\S(M') \subset \cone(M') \cap
\Z^{n+1}$, and if the two sets are equal we call $\S(M')$ \emph{normal}. The monoid $\S(M')$ is normal 
if and only if $\left\{ \u \in \Z^n : (k, \u) \in \cone(M')\right\} = k\P_M \cap \Z^n$ for all $k \in \Z_{> 0}$.

Now let $K$ be any field and let
$K[\x]=K[x_1,\dots, x_r,x_{r+1}]$ be the ring of polynomials with
coefficients in $K$. A monomial of $K[\x]$
is a product of powers of variables, $\x^\u=x_1^{u_1}x_2^{u_2}\dots
x_{r+1}^{u_{r+1}}$ where $\u=(u_1,u_2,\dots ,u_{r+1})$ is the
exponent vector. Similarly, we let $T=K[s, t_1,t_1^{-1},\dots
,t_n,t_n^{-1}]$ be the 
Laurent polynomial ring where the monomials can have exponent vectors with
negative coordinates (except in the first position).  The monoid algebra
$K[M'] = K[s\t^{\a_1},s\t^{\a_2},\ldots, s\t^{\a_r},s]$ is the subalgebra of $T$ consisting of 
$K$-linear combinations of monomials $s^k\t^\u$ where $(k,\u) \in \S(M')$. The ring homomorphism 
$\psi : K[\x]  \rightarrow T$ defined by 
$\psi (x_i)= s\t^{\a_i}$ for $i=1,\ldots, r$ and $\psi(x_{r+1}) = s$ is a surjection onto $K[M']$. 
 Therefore $K[\x]/I_{M'} \cong K[M']$ where $I_{M'}$, known as the \emph{toric ideal} of $M'$,  is the
kernel of $\psi$.  The monoid algebra $K[M']$ is graded where the degree of the
monomial $s^k \t^{\u}$ is ~$k$. The toric ideal $I_{M'}$ is homogenous with respect to the same grading.

\begin{definition}
The \emph{Hilbert series} of $K[M']$ is the generating function
\[H(K[M'];x):= \sum_{k \geq 0} \dim(K[M']_k) \, x^k,\]
where $K[M']_k$ is the $K$-vector space
 of the monomials in $K[M']$ of degree $k$.
\end{definition}

The following theorem is a standard result from algebraic geometry.

\begin{theorem}\cite{eisenbud}\label{hilbertgrowth}
The Hilbert series  of $K[M']$ can be written as
\[
H(K[M']; x) \, = \, \frac{h(x)}{(1-x)^{d+1}},
\]
where $h(x)$, the \emph{$h$-polynomial} of $K[M']$, is a polynomial of degree at most $d=\mathrm{rank}(M)$.
\end{theorem}

By our construction, the set of vectors $\{u : s^k \t^{\u}  \in
K[M']\}$ are in bijection with the set of vectors in $\Lat_M$ with word length
at most  $k$; that is,
\[\dim(K[M']_k) = \sum_{i=0}^k \growth(i).\] 
This observation gives us the following.

\begin{proposition}\label{prop:hilbert}
The $h$-polynomial of $K[M']$ is precisely the coordinator
polynomial of the growth series of $\Lat_M$.
\end{proposition}

It follows that computing the coordinator polynomials of 
$A_n$, $C_n$, and $D_n$ is equivalent to computing the $h$-polynomials of the corresponding monoid
algebras. We will use this point of view in Section \ref{sec:Cn}.

Finally, we show that the $h$-polynomial of $K[M']$ is essentially the $h$-polynomial of any unimodular
triangulation of $\cone(M')$. 
\begin{theorem} Let $\Gamma$ be a unimodular triangulation of $\cone(M')$, and denote by  $K[\Gamma]$  the Stanley-Reisner ring of $\Gamma$ as a simplicial complex. Then 
$H(K[M']; x) =  H(K[\Gamma]; x)$; hence, the $h$-polynomial of $K[M']$ is equal to $t^d h_\Gamma(\frac{1}{t})$.
\end{theorem}
\begin{proof} Since $\Gamma$ is unimodular, the monoid $\S(M')$ is the disjoint union of all sets of the form
\[ \left\{ \left( \begin{array}{c} k \\ \u \end{array} \right) +\sum_{i\in \sigma}
n_i  \left( \begin{array}{c} 1 \\ \a_i \end{array} \right)   
\,\,: \,\,
n_i \in \N \right\},
\]
where $\sigma \subset \{1,\ldots, r+1\}$ is a face of $\Gamma$ with $|\sigma| = k$ and $\u = \sum_{i \in \sigma} \a_i$.
This means that 
\[ H(K[M']; x) = \sum_{\sigma \in \Gamma} \frac{x^{|\sigma|}}{(1-x)^{|\sigma|}}, \]
and the above expression is precisely $H( K[\Gamma]; x)$.
\end{proof}

\section{The Lattice $A_n$}

We now take a closer look at the root lattice $A_{n}$; it
is the subgroup of $\Z^{n+1}$ given by $A_{n} =
\left\{\x\in \Z^{n+1}\,|\, \sum_{i=0}^{n} x_i = 0 \right\}.$

\begin{proposition}
  The lattice $A_{n}$ is generated as a monoid by the root system of the Coxeter group $A_{n}$; that is, the set of vectors $\M_{A_{n}} = \{\e_i-\e_j \, : \, 0 \leq i \neq j \leq n+1\}$. 
\end{proposition}

\begin{proof}
Define $|\x| = \sum |x_i|$. Any $\x \in A_n$ with $|\x|>0$ must have a positive entry $x_i$ and a negative entry $x_j$. Subtracting $\e_i-\e_j$ from $\x$ gives the vector $\y = \x - (\e_i-\e_j)$, which is also in $A_n$ and satisfies $|\y| < |\x|$. Iterating this process yields a way to write $\x$ as a non-negative integer combination of $\M_{A_n}$.
\end{proof}

The polytope $\P_{\M_{A_n}}$ is the root polytope of the lattice $A_n$, and we will denote this polytope 
by $\P_{A_n}$. Each root $\e_i-\e_j$ is a vertex of $\P_{A_n}$ since it uniquely maximizes the functional $x_i-x_j$. It will be convenient to let 
$\v_{ij}=\e_i - \e_j$ and organize these vectors in the $(n+1) \times (n+1)$
matrix~$\V_{n}$ whose entries are $\v_{ij}$ for $i \neq j$ and 0 if
$i=j$.

\begin{example}
The root polytope $\P_{A_3}$ 
can be written
as $\conv(\V_3)$, the convex hull of the entries of
\[\V_{3}={\scriptsize \left(
                    \begin{array}{cccc}
                      0 & \v_{01} & \v_{02} & \v_{03}\\
                \v_{10} & 0 & \v_{12} & \v_{13}\\
                   \v_{20} & \v_{21} & 0 & \v_{23}\\
                     \v_{30} & \v_{31} & \v_{32} & 0 \\
                                     \end{array}
                  \right)}.
\]

The root polytope $\P_{A_3}$ can be obtained by joining the midpoints of the edges of a cube, as shown in Figure \ref{rootA3}. To see this, let $\a_1 =(\frac32, -\frac12, -\frac12, -\frac12)$ and define $\a_2, \a_3, \a_4$ analogously; these are the vertices of a regular tetrahedron centered at the origin which lies on the hyperplane $x_1+x_2+x_3+x_4= 0$ of $\R^4$. These vectors, together with their negatives, are the vertices of a $3$-cube. The midpoints of the edges of this cube are the vectors $\frac12(\a_i-\a_j) = \v_{ij}$. In the diagram, $i$, $-i$, and $ij$ represent $\a_i$, $-\a_i$, and $\v_{ij}$, respectively.\end{example}

\begin{center}
\begin{figure}[htbp]
\centering
\includegraphics[scale=.65]{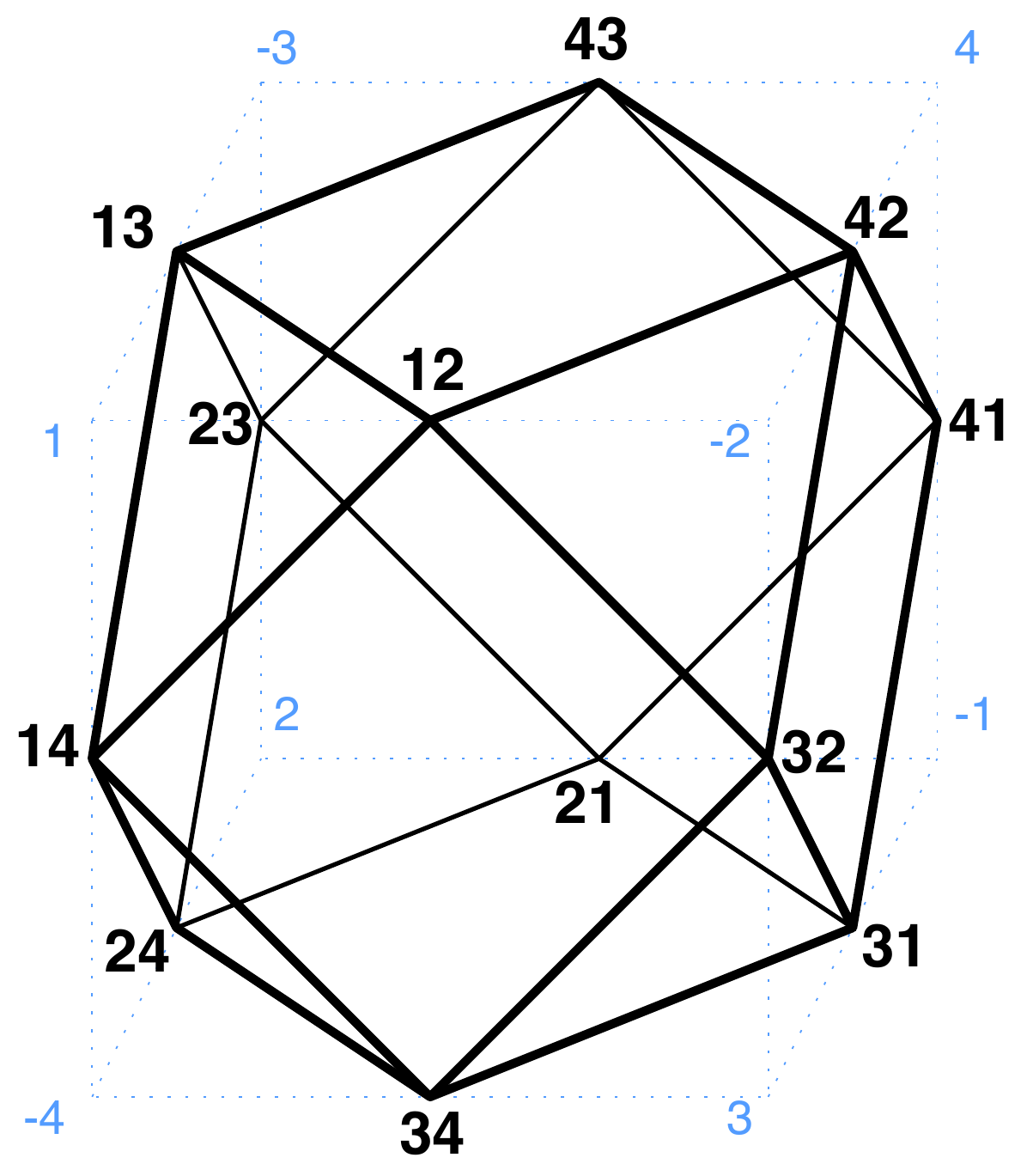}
\caption{The root polytope $\P_{A_3}$.}
\label{rootA3}
\end{figure}
\end{center}

For a finite set $A$ let $\Delta_A$ be the standard $(|A|-1)$-dimensional simplex whose vertices are
the unit vectors in $\R^{|A|}$. The following proposition summarizes several useful properties of the root polytope $P_{A_n}$.

\begin{proposition}\label{hyperplanedesc}
The polytope  $\P_{A_n}$ is an $n$-dimensional polytope in $\R^{n+1}$ which is contained in the hyperplane $H_0 = \{\x \in \R^{n+1} :\sum_{i=0}^{n} x_i =0\}$. 
It has $(n-1)n(n+1)$ edges, which are of the form
  $\v_{ij}\v_{ik}$ and $\v_{ik}\v_{jk}$ for $i,j,k$ distinct. 
  It has $2^{n+1}-2$
  facets, which can be labelled by the proper subsets $S$ of $[0,n] := \{0,1,\ldots, n\}$. The facet $\F_S$ 
is defined by the hyperplane
  \[
     H_S := \left\{\x \in \R^{n+1} : \,\,\sum_{i\in S} x_{i} = 1
     \right\},
  \] 
and it is congruent to
the product of simplices $\Delta_S \times \Delta_T$, where $T =
[0,n]-S$. 
The only lattice points in $\P_{A_n}$ are its vertices and the origin.

\end{proposition}

\begin{proof} The first statement is clear. The edges  $\v_{ij}\v_{ik}$ and
  $\v_{ik}\v_{jk}$ are maximized by the functionals $x_i-x_j-x_k$ and
  $x_i+x_j-x_k$, respectively. To see that we cannot have an edge whose
  vertices are $\v_{ij}$ and $\v_{kl}$ with $i,j,k,l$ distinct, note
  that $\v_{ij}+\v_{kl} = \v_{il}+\v_{kj}$; so any linear functional $f$ satisfies that
  $f(\v_{ij}) + f(\v_{kl}) = f(\v_{il}) +  f(\v_{kj})$ and cannot be
  maximized precisely at this presumed edge. Similarly, we cannot have an edge with vertices
  $\v_{ij}$ and $\v_{ki}$ since $\v_{ij}+\v_{ki} = \v_{lj} + \v_{kl}$ for
  any $l$ distinct from $i,j,k$.

The vertices of $\P_{A_n}$ that lie on $H_S$ are those of the form
$\v_{ij}$ for $i \in S$ and $j \notin T$; these clearly form a polytope
isomorphic to the product of simplices $\Delta_S \times \Delta_{T}$. Since
this polytope has codimension one in $\P_{A_n}$, it is indeed a facet. 
Now consider any facet $\F$ of $\P_{A_n}$ defined by a
functional $f$. By the previous paragraph, every edge of $\F$ has
direction of the form $\e_i-\e_j$ for some $i \neq j$; and the fact that this edge
lies on $\F$ implies that
$f(\e_i) = f(\e_j)$. Doing this for $n-1$ linearly independent edges of
$\F$, we get $n-1$ independent equalities among the $f(\e_i)$, showing that
there can only be two different values among $f(\e_0), \ldots, f(\e_n)$. Since
$\P_{A_n}$ lies on $H_0$ we can assume that the smaller value is $0$, and
by rescaling we can make the larger value $1$. It follows that $\F$ is
one of the facets already described.

Since $\P_{A_n}$ is contained in the sphere of radius $\sqrt{2}$ centered at $0$, it can only contain lattice points of the form $0, \pm \e_i$, or $\pm \e_i \pm \e_j$. Of these, the only ones on the hyperplane $H_0$ are the origin and the vertices of $\P_{A_n}$.
\end{proof}

Now we will construct a specific unimodular triangulation of $\P_{A_n}$ (equivalently, of $\cone(M_{A_n}')$). The combinatorial structure of this triangulation will allow us to enumerate its faces.

\begin{proposition}\cite{seashore}\label{uniproof}
  The matrix $M_{A_n}$  is totally unimodular; \emph{i.e.}, every
  square submatrix of~$M_{A_n}$ has determinant $0, \, +1$, or $-1$. The same is true for $M_{A_n}'$.
\end{proposition}

\begin{corollary}\label{unimodAn} 
Let $\T$ be an arbitrary triangulation of the boundary of $\P_{A_n}$. Coning over $\T$ from the origin gives rise to a unimodular triangulation of $\P_{A_n} \cup 0$.
\end{corollary}

In light of Corollary \ref{unimodAn} we need a  triangulation of the boundary of $\P_{A_n}$. 
Such a triangulation will be obtained as a {\em pulling triangulation}, also known as a {\em reverse 
lexicographic triangulation}. 
Let $\A=\{a_1, \ldots, a_r\}$ be a collection of points in $\R^n$, totally ordered by
$a_1 \prec a_2 \prec \dots \prec a_r$, and let $\P_{\A}$ be the convex hull of $\A$. We define the 
faces of $\A$ to be the subsets of $\A$ that lie on the faces of the polytope $\P_{\A}$.
%
%
\begin{definition}
  A  pulling triangulation $\Gamma_{pull}(\A)$ is defined
  recursively as follows:
\begin{itemize}
  \item If $\A$ is affinely independent, then $\Gamma_{pull}(\A)
  = \{\A\}.$
  \item Otherwise, \[\Gamma_{pull}(\A) = \bigcup_\F
  \{\{a_1\} \cup \, \G : \G \in \Gamma_{pull} (\F)\}\] where the
  union is taken over all facets $\F$ of $\A$ not containing $a_1.$
\end{itemize}
\end{definition}

\begin{definition}
The staircase triangulation $\Gamma$ of $\P_{A_n}$ is the pulling triangulation of the set $\M_{A_n} \cup \0$ under the ordering
\[
\0 \prec \v_{01} \prec \v_{02} \prec \cdots \prec \v_{0,n} \prec \v_{10} \prec \v_{12} \prec \cdots \prec \v_{1,n} \prec \cdots \prec \v_{n,0} \prec \v_{n, n-1}
\]
\end{definition}

Since the  origin is pulled first, $\Gamma$ is the cone over a triangulation of the boundary of $\P_{A_n}$, and it suffices to understand how each facet $\Delta_S \times \Delta_T$ of $\P_{A_n}$ gets triangulated. Fortunately, the restriction of $\Gamma$ to this facet is the well-understood 
\emph{staircase triangulation} of a product of two simplices \cite{GKZ}. The vertices of $\Delta_S \times \Delta_T$ correspond to the entries of the $S \times T$ submatrix of the matrix $\V_n$, and the maximal simplices of the staircase triangulation of this facet correspond to the ``staircase" paths that go from the top left to the bottom right corner of this submatrix taking steps down and to the right.
We will let $\Gamma=\Gamma(\P_{A_n}), \partial \Gamma = \Gamma(\partial\P_{A_n}),$ and $\Gamma(\Delta_S \times \Delta_T)$ denote the staircase triangulation of $\P_{A_n}$, its restriction to the boundary $\partial \P_{A_n}$, and its restriction to the facet $\Delta_S \times \Delta_T$, respectively.

\begin{example} The facet $\F_{046} \cong \Delta_{046} \times \Delta_{1235}$ of the root polytope $\P_{A_6}$ is the convex hull of the vertices in the following submatrix of $\V_6$:
\[
\left(
  \begin{array}{cccc}
  \v_{01} & \v_{02} & \v_{03}& \v_{05}\\
    \v_{41} & \v_{42} & \v_{43} & \v_{45} \\
    \v_{61} & \v_{62} & \v_{63} & \v_{65}\\
      \end{array}
\right).
\] 
There are $10$ maximal cells in $\Gamma(\Delta_{046} \times \Delta_{1235})$, corresponding to the ${5 \choose 2} = 10$ staircase paths from $\v_{01}$ to $\v_{65}$. The simplex with vertices $\{\v_{01},\v_{02},\v_{42},\v_{43},\v_{45},\v_{65}\}$ is one of these cells; it corresponds to the staircase
\[
\left[
  \begin{array}{cccc}
  \v_{01} & \v_{02} & - & -\\
    - & \v_{42} & \v_{43} & \v_{45}     \\
  - & - & - & \v_{65}\\
      \end{array}
\right].
\] 
\end{example}

The facet $\Delta_S \times \Delta_T$ is subdivided into ${|S|+|T| -2 \choose |S|-1}$ simplices. It follows that the number of full-dimensional simplices of $\Gamma$ is
\[
f_n(\Gamma) = f_{n-1}(\partial\Gamma) = \sum_{k=1}^{n} {n+1 \choose k} {n-1 \choose k-1} = {2n \choose n}.
\]
The following gives a characterization of all faces of $\partial \Gamma$.

\begin{proposition}\label{mfaceconditions}
The $(m-1)$-dimensional faces of the staircase triangulation $\partial\Gamma$ 
are given by subsets $\{v_{i_1 j_1}, \ldots, v_{i_mj_m}\}$ of the set of vertices such that:
\begin{enumerate}
  \item $0 \leq i_1 \leq i_2 \leq \ldots\leq i_{m}\leq n$
  \text{and} $\,0\leq j_1 \leq j_2 \leq \ldots\leq j_{m}\leq n$,
\item $(i_s,j_s) \neq (i_t,j_t)$ for $s \neq t$, and
  \item $i_s\neq j_t,$\,\,$1\leq s,t\leq m$.
\end{enumerate}
\end{proposition}

\begin{proof}
This is straightforward from the definitions. The third condition guarantees that $\v_{i_1 j_1}, \ldots, \v_{i_mj_m}$ are vertices of some (not necessarily unique) facet $\Delta_{S} \times \Delta_{T}$, while the first two conditions guarantee that they form a subset of some staircase path in the corresponding $S \times T$ submatrix of $\V_n$.
\end{proof}

In the matrix $\V_n$, we can see a face as a sequence of positions that (1) moves weakly southeast, (2) never stagnates,  and (3) never uses a row and a column of the same label. This is illustrated in Figure \ref{triangA3}.

\begin{center}
\begin{figure}[htbp]
\centering
\includegraphics[scale=.6]{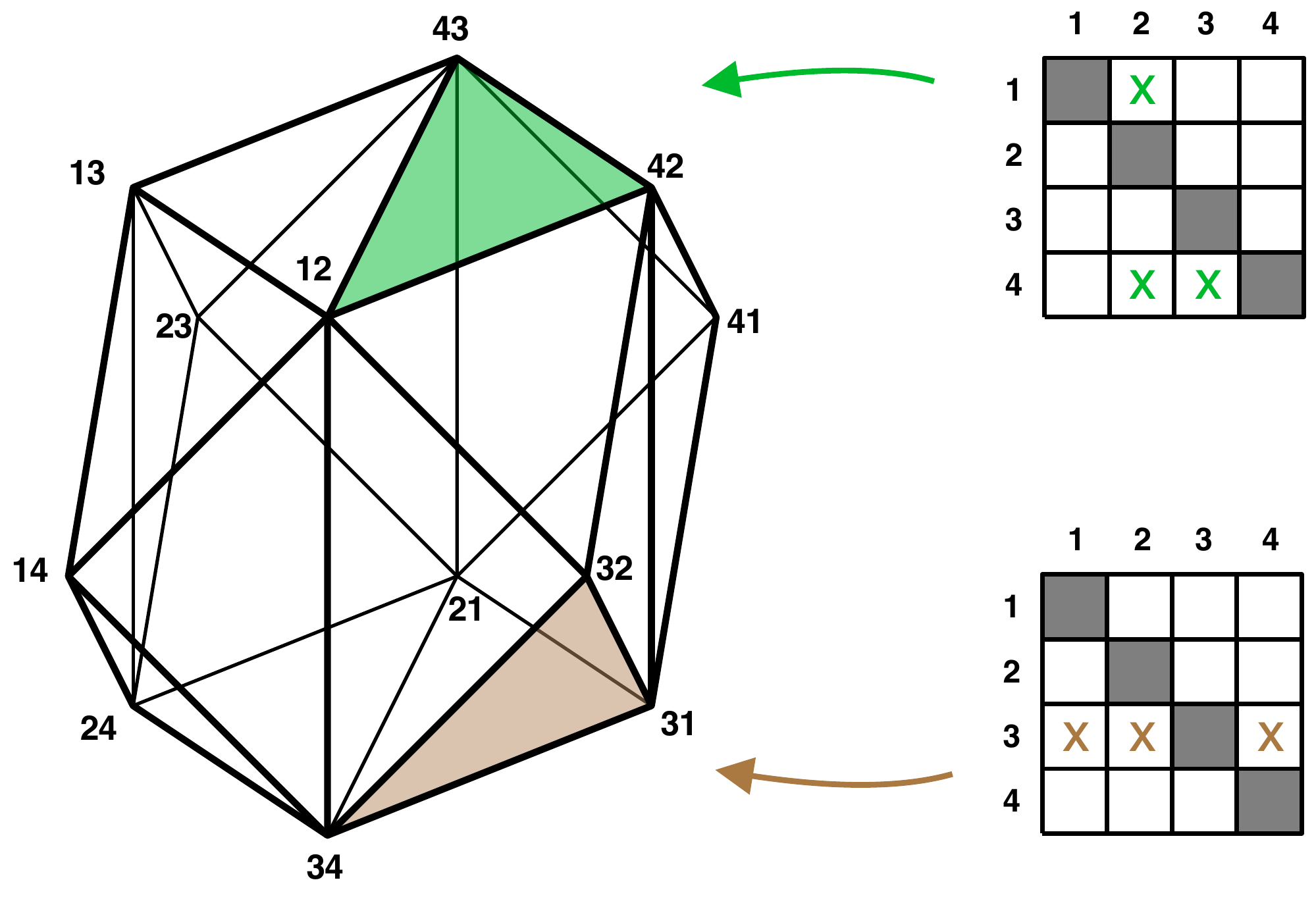}
\caption{The staircase triangulation of the root polytope $\P_{A_3}$.}
\label{triangA3}
\end{figure}
\end{center}

\begin{theorem}
The $f$-vector of the 
reverse lexicographic triangulation $\partial\Gamma$ of the boundary of $P_{A_n}$
is given by
\begin{equation*}
f_{m-1}(\partial \Gamma) = f_m(\Gamma) =  \binom{n+m}{m,m,n-m}
\end{equation*}
for $0 \leq m \leq n$.
\end{theorem}

\begin{proof}
The problem of counting the $m$-faces of $\Gamma(P_{A_n})$ has
been reduced to counting the possible choices for indices $(i_1, j_1), \ldots, (i_m, j_m)$
that meet the conditions in Proposition \ref{mfaceconditions}. We find it convenient to view these as $m$ positions in the matrix $\V_n$.

Consider the choices that occupy exactly $a$ columns and $b$ rows. There are ${n+1 \choose a,b,n+1-a-b}$ choices of rows and columns satisfying condition (3). Once we have chosen these, let us focus our attention on the $a \times b$ submatrix that remains; condition (3) guarantees that there are no zero entries in this matrix. Inside this submatrix, the vertices $(i_1, j_1), \ldots, (i_m,j_m)$ form a path that never moves up or to the left and touches every row and column. This path must go from the top left to the bottom right corner of the submatrix using the steps $(1,0), (0,-1)$, and $(1,-1)$. The total number of steps is $m-1$ and the distances covered by the path in the horizontal and vertical directions are $b-1$ and $a-1$, respectively. Therefore, the number of steps of the form $(1,0), (0,-1)$, and $(1,-1)$ must be $m-a, m-b,$ and $a+b-m-1$, respectively. There are ${m-1 \choose m-a,m-b,a+b-m-1}$ paths of this form. 
\begin{center}
\begin{figure}[htbp]
\centering
\includegraphics[scale=.5]{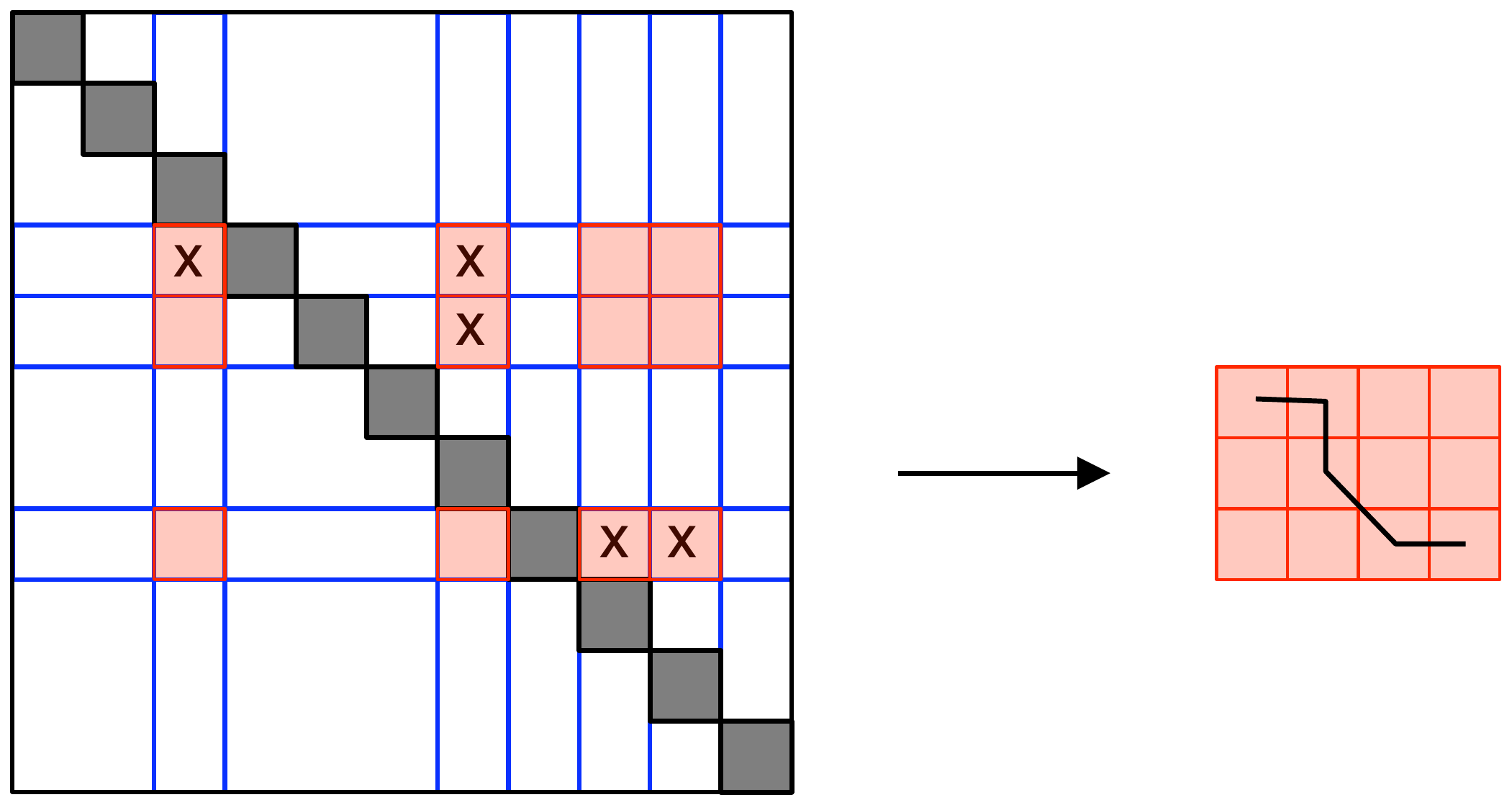}
\caption{Going from the label of a cell of $\partial \Gamma$ to a path in a rectangle with steps $(1,0), (0,-1)$ and $(1,-1)$.}
\label{fig:staircase}
\end{figure}
\end{center}
Figure \ref{fig:staircase} shows how to convert a partial staircase into a path from the top left to the bottom right corner of a rectangle, using steps of the form $(1,0), (0,-1)$ and $(1,-1)$. First we select the rows and columns used by the partial staircase, and delete all the other ones. The cells of the staircase form a path in the resulting rectangle.

It follows that
\begin{eqnarray*}
f_{m-1}(\partial\Gamma) &=& \sum_{a,b} {n+1 \choose a,b,n+1-a-b} {m-1 \choose m-a,m-b,a+b-m-1}\\
{n+m \choose n+1} f_{m-1}(\partial \Gamma)&=& 
\sum_{a,b} {n+m \choose n+1} {n+1 \choose a,b,n+1-a-b} {m-1 \choose m-a,m-b,a+b-m-1}\\
&=&  \sum_{a,b} {n+m \choose a,b,n+1-a-b, m-a,m-b,a+b-m-1}\\
&=&  \sum_{a,b} {n+m \choose m, m, n-m} {m \choose a} {m \choose b} {n-m \choose n+1-a-b}\\
&=& {n+m \choose m, m, n-m} {n+m \choose n+1}\\
\end{eqnarray*}
as desired. \end{proof}

\begin{theorem}\label{ftoh}
  The coordinator polynomial for the growth series of the lattice $A_{n}$
  generated as a monoid by $\M_{A_{n}}=\{\e_i-\e_j,:\, 0 \leq i,j \leq n+1
  \,\,\text{with}\,\,i \neq j \}$ is
  \[
  h_{A_n}(x) = \sum_{k=0}^{n} {n \choose k}^2 x^k.
  \]
\end{theorem}

\begin{proof}
We first compute $h_{\partial \Gamma}(x)$. By definition $h_{\partial \Gamma}(x+1) = 
f_{\partial \Gamma}(x)$, which is
\begin{eqnarray*}
\sum_{m=0}^n {n+m \choose m,m,n-m}x^{n-m} =  \sum_{l=0}^n  {n \choose l}{2n-l \choose n}x^l &=& \sum_{l=0}^n \sum_{k=l}^n {n \choose l} {n \choose k} {n-l \choose n-k} x^l = \\
\sum_{l=0}^n \sum_{k=l}^n {n \choose k}^2 {k \choose l} x^l = \sum_{k=0}^n \sum_{l=0}^k {n \choose k}^2 {k \choose l} x^l &=& \sum_{k=0}^n {n \choose k}^2(x+1)^k. 
\end{eqnarray*}
This shows that $h_{\partial \Gamma}(x) = \sum_{k=0}^n {n \choose k}^2 x^k$, and since this 
polynomial is palindromic we get $h_{A_n}(x) = x^n h_{\partial \Gamma}(x^{-1}) = h_{\partial \Gamma}(x)$. 
\end{proof}

\section{The Lattice $C_n$} \label{sec:Cn}
The root lattice $C_n \subset \R^n$ is defined by $C_n:= \{\x \in \Z^n :
\sum_i x_i\,\,\text{is even}\}.$
\begin{proposition}\label{MCn}
The root lattice $C_n$ is a rank $n$ lattice generated as a monoid by the set
$\M_{C_n} =\{\pm 2\e_i : 1 \leq i \leq n \, ;\,  \pm \e_i \pm \e_j: 1  \leq i \neq j \leq n\}$.
\end{proposition}
\begin{proof}
Consider any $\x \in C_n$. By subtracting from it the appropriate nonnegative multiples of $2\e_i$ or $-2\e_i$, we obtain a vector $\y \in C_n$ where each $y_i$ equals $0$ or $1$. The number of coordinates of $\y$ equal to $1$ is even, so $\y$ is a sum of vectors of the form $\e_i+\e_j$ with $i \neq j$. \end{proof}
As in the earlier section, we define $\P_{C_n}= \conv(\M_{C_n})$. 
The root polytope $\P_{C_n}$ is the cross polytope (which we will define next) dilated by a factor of two, 
$\P_{C_n}= 2\Diamond_n$.

\begin{definition} The \emph{crosspolytope} $\Diamond_n$ in $\R^n$
is given by the facet and vertex descriptions
\[\Diamond_n:=\left\{\x \in \R^n : \sum_{i=1}^n |x_i| \leq 1\right\} =
\conv(\{\e_1, -\e_1,\dots,\e_n,-\e_n\}).\]\end{definition}
\begin{remark} In the rest of this section we will compute the growth series of $\Lat_{\M_{C_n}}$ 
essentially by considering lattice points in the integer dilates of $\P_{C_n}$. However, when we say
``lattice points" or ``unimodular triangulation," we mean these words {\em with respect to the lattice} 
$\Lat_{\M_{C_n}}$. For instance, in this sense, the only interior lattice point of $\P_{C_n}$ is the origin.
\end{remark}
In Section \ref{sec:reprise} we will derive a unimodular triangulation of $\P_{C_n}$ from a triangulation of the root polytope $\P_{D_n}$, and use it to give a proof of Theorem \ref{maintheorem} in the $C_n$ case. Such a triangulation is obtained from a triangulation of the facets of $\P_{D_n}$ by forming the cone of this boundary triangulation from the origin.
One can also construct  a specific pulling triangulation where the origin is pulled first; for the details we refer the reader to Section 5.2 in \cite{seashore}. Here we give a different proof; we start by giving a simple description of the lattice points on the faces of the root polytope $\P_{C_n}$.

\begin{proposition} The set of lattice points on any $(k-1)$-face of $\P_{C_n}$ for $k=0,\ldots, n$ is affinely isomorphic to the set of lattice points in the simplex
\[
V(2,k) := \left\{ \u \in \Z_{\ge 0}^k \, : \: u_1 + u_2 + \cdots + u_k = 2 \right\}. 
\]
\end{proposition}

\begin{proof}
This is immediate from the fact that crosspolytopes are regular and simplicial.
\end{proof}

The Hilbert series of $K[V(2,k)]$ is given in \cite[Corollary 2.6]{katzman} as 
\[H(K[V(2,k)], x) = \frac{h_k(x)}{(1-x)^k} = \frac{\sum_{i=0}^k {k \choose 2i} x^i}{(1-x)^k}. \]
Now we can compute the growth series of $C_n$ as an inclusion-exclusion count of Hilbert 
series of above kind for all dimensions, namely
\[ \sum_{j=0}^n (-1)^{n-j} f_{j-1} \frac{h_j(x)}{(1-x)^j} \]
where $f_{j-1}$ is the number of $(j-1)$-dimensional faces of the cross polytope $\Diamond_n$.
Using the duality between the cross polytope and the $n$-dimensional hypercube we know 
that $f_{j-1} = {n \choose j} 2^j$. Substituting this in the above series and writing with a common
denominator we get 
\[ \frac{1}{(1-x)^n} \sum_{j=0}^n (-1)^{n-j} (1-x)^{n-j}  {n \choose n-j} 2^j  \sum_{i=0}^j {j \choose 2i} x^i.\]
The numerator of this series is the coordinator polynomial we are after, and note that this polynomial
consists of the whole power terms of
\[ \sum_{j=0}^n (-1)^{n-j} (1-x)^{n-j}  {n \choose n-j} 2^j  (1 + \sqrt{x})^j.\]
Using the binomial theorem, the above polynomial is equal to 
$((x-1) + (2 + 2\sqrt{x}))^n = (1 + \sqrt{x})^{2n}$. This leads us to our main theorem in this section.
 
\begin{theorem}
The coordinator polynomial for the lattice $C_n$ generated as a
monoid by the standard generators $\M_{C_n} = \{\pm \e_i \pm \e_j: 1
\leq i,j \leq n\}$ is given by
\[h_{C_n}(x) = \sum_{k=0}^n \binom{2n}{2k} x^k.\]
\end{theorem}

From this, we use the method of Theorem \ref{ftoh}
to extract the corresponding
$f$-polynomial
\[
 f_{C_n}(x) = 
  \sum_{m=0}^n
  \frac{n \, 2^{2m}}{n+m}\binom{n+m}{2m} x^{n-m}
\]
%
proving the statement in Theorem \ref{fpolys}.

\section{The lattice $D_n$}

The root lattice $D_n$ is defined by $D_n:= \{x\in\ZZ^n: \sum_i x_i
\text{ is even}\}$. Note that this lattice is the same as $C_n$. The only difference is in the set 
of generators.

\begin{proposition}
  The root lattice $D_n$ is generated by $\M_{D_n}=\{\pm
  \e_i \pm \e_j: 1 \leq i \neq j \leq n\}$  as a monoid. \hfill$\Box$
\end{proposition}

\begin{proof}
Observe that $2\e_i = (\e_i+\e_j) + (\e_i-\e_j)$ and invoke Proposition \ref{MCn}.
\end{proof}

Let $\Delta_{2,n} = \conv\{\e_i+\e_j:1\le i\ne j\le n\}$ be the $(n-1)$-dimensional second hypersimplex  in~$\RR^n$.

\begin{proposition}
  The root polytope $\P_{D_n}=\conv\M_{D_n}$ has dimension~$n$,
  $2n(n-1)$~vertices, $2^n$~hypersimplex facets $\Delta_{2,n}$, and
  $2n$~cross-polytope facets~$\Diamond_{n-1}$.
\end{proposition}

\begin{proof}
  The first two statements are immediate. Next, we claim that the
  facet-defining inequalities come in two families, $\sum_{i=1}^d
  \sigma_i x_i \le 2$ and $\sigma_i x_i\le 1$, where
  $\sigma_i\in\{\pm1\}$.
  To verify that they all describe facets of $\P_{D_n}$, it suffices
  to consider the case $\sigma_i=1$ for all $i$. The vertices
  in~$\P_{D_n}$ with $x_1+\dots+x_d=2$ are just the
  permutations of $(1,1,0,\dots,0)$, which by definition are the
  vertices of $\Delta_{2,n}$. For the second family, we can
  additionally assume $i=1$, so that the set of vertices in $\P_{D_n}$
  satisfying $x_1=1$ is exactly
  $(1,\pm1,0,\dots,0),\dots,(1,0,\dots,0,\pm1)$, the set of vertices
  of an $(n-1)$-dimensional cross-polytope $\Diamond_{n-1}$.

  Convincing ourselves that we didn't miss any facets of $\P_{D_n}$
  can be done directly, but it also follows quickly using the fact
  that   $\P_{D_n}$ arises from the \emph{Wythoff
    construction}~\cite{Coxeter63} associated to the following
  diagram:
\begin{center}
  \begin{picture}(100,15)(0,-5)
    \thicklines
    \put(0,0){\circle{5}}
    \put(6,2){\text{\small 4}}
    \put(2.5,0){\line(1,0){10}}
    \put(15,0){\circle{5}}
    \put(17.5,0){\line(1,0){10}}
    \put(30,0){\circle{5}}
    \put(32.5,0){\line(1,0){4}}

    \put(42,0){\circle{1}}
    \put(45,0){\circle{1}}
    \put(48,0){\circle{1}}

    \put(53.5,0){\line(1,0){4}}
    \put(60,0){\circle{5}}
    \put(62.5,0){\line(1,0){10}}
    \put(75,0){\circle{5}}
    \put(77.5,0){\line(1,0){10}}
    \put(90,0){\circle{5}}
    \put(90,0){\circle{10}}
    \put(92.5,0){\line(1,0){10}}
    \put(105,0){\circle{5}}
  \end{picture}
\end{center}
We can read off the facets of $\P_{D_n}$ from this diagram by forming
connected subgraphs with $n-1$ vertices that contain the ``ringed''
node. There are exactly two such subgraphs,
corresponding to cross-polytopes $\Diamond_{n-1}$ and hypersimplices
$\Delta_{n,2}$, respectively. The counting method detailed in
\cite{Coxeter63} now assures us that we have accounted for all facets
of $\P_{D_n}$.
\end{proof}
\medskip To calculate the generating function of the $f$-vector of our 
triangulation $\triang\partial\P_{D_n}$ of the boundary of $\P_{D_n}$,
we split 
\[
   \triang\partial \P_{D_n} 
   \ = \
   \triang\P_{D_n,\total\interior\Diamond_{n-1}}
   \ \cup \ 
   \triang\P_{D_n,\total\Delta_{2,n}}
\]
into the disjoint union 
of all faces in the \emph{interior} of a cross-polytope facet,
respectively \emph{all} faces
contained in a hypersimplex facet.

\subsection{Triangulating the cross-polytope facets}

We start with $\triang \P_{D_n,\total\interior\Diamond_{n-1}}$. Since
cross-polytopes are simplicial, the way in which we choose to
triangulate them will not affect the triangulations of the
hypersimplex facets, and the entire boundary of the cross-polytope
will be contained in the triangulation. More precisely, there are
$\binom{m}{k+1} 2^{k+1}$ faces of dimension~$k$ in the boundary
complex of an $m$-dimensional cross-polytope; we include the empty
face by allowing $k=-1$.  These numbers assemble into the generating
function
\[
  f_{\partial\Diamond_m} (x)
  \ = \ 
  \sum_{k=-1}^{m-1} \binom{m}{k+1}2^{k+1} x^k
  \ = \
  \frac1x \sum_{k=0}^m \binom{m}{k} (2x)^k
  \ = \
  \frac{(1+2x)^m}{x}.
\]
We now extend the boundary to a unimodular triangulation of
$\Diamond_{n-1}$. However, just as in Section~\ref{sec:Cn}, we must be
careful to only use lattice points \emph{from our root lattice};
for~$D_n$ this means that we must exclude the centroid of
each~$\Diamond_{n-1}$ from the triangulation, because it has odd
coordinate sum.

Instead, setting $m=n-1$, we need to triangulate the interior of
$\Diamond_{m}$ without using any new vertices.  There is still not much
choice in the matter, since all such triangulations are
combinatorially isomorphic; they are determined by the choice of a
diagonal connecting two opposite vertices, say $v_{2m-1}$ and
$v_{2m}$. The faces of the triangulation of $\Diamond_m$ are then
either faces of the boundary $\partial\Diamond_m$, or joins
$\{v_{2m-1},v_{2m}\}\star F$, where $F$ is a face of the
``equatorial'' $\partial\Diamond_{m-1}$. Here we include the case
$F=\emptyset$, which we take to yield the diagonal
$\conv\{v_{2m-1},v_{2m}\}$ itself.  However, we \emph{exclude} the
faces in the boundary $\partial\Diamond_m$ from our count, because
these will be included in the generating function of the triangulation
of the hypersimplex facets.  Thus, the $f$-vector of interior faces of
our triangulation of $\Diamond_{n-1}$ is
\[
   f_{\interior\Diamond_{n-1}}(x) 
   \ = \
   x^2 f_{\partial\Diamond_{n-2}}(x)
   \ = \ 
   (1+2x)^{n-2},
\]
and these values assemble into the following generating function for the
number of interior faces of~$\Diamond_{n-1}$:
\[
   F_{\interior\Diamond_{n-1}}(z,x) 
   \ = \
   \sum_{n\ge2} f_{\interior\Diamond_{n-1}}(x) z^n
   \ = \
   \sum_{n\ge2} x (1+2x)^{n-2} z^n 
   \ = \
   - \frac{x z^2}{2xz+z-1}.
\]
Because $\P_{D_n}$ has $2n$ cross-polytope facets, we obtain the total
count of such interior faces as
\[
   F_{D_n,\total\interior\Diamond_{n-1}} (z,x) \ = \ 
   2z \frac{\partial}{\partial z}  F_{\interior\Diamond_{n-1}}(z,x)
   \ = \
   -\frac{2xz^2(2xz+z-2)}{(2xz+z-1)^2}.
\]
For the $f$-vector, this means
\[
   f_k^{\total\interior\Diamond_{n-1}} \ = \
   n \, 2^{k}\binom{n-2}{k-1}.
\]
Note that there is no double-counting here, because the interior faces
of the triangulations of the cross-polytope facets are disjoint.

\subsection{Triangulating the hypersimplex facets}

Our next job is to calculate the number $f_k(D_n,\total\Delta_{2,n})$
of $k$-dimensional faces of a unimodular triangulation of the
hypersimplex facets of~$\P_{D_n}$.

The unimodular triangulation of the hypersimplex $\Delta_{2,n}$ that
we will use has an interesting history. Two such triangulations were
constructed independently by Stanley~\cite{Stanley77} in~1977 and by
Sturmfels~\cite{sturmfels} in~1996. However, in~2006, Lam and
Postnikov~\cite{Lam-Postnikov06} showed that these triangulations are
in fact the same, and gave two more independent descriptions.

Let $f_k^{(n)}$ count the number of $k$-faces of this ``standard''
triangulation of the $(n-1)$-dimensional hypersimplex $\Delta_{2,n}$.
These numbers were calculated by Lam and Postnikov, and are given
implicitly in \cite[Remark~5.4]{Lam-Postnikov06} via the generating
function
\begin{eqnarray}
   F_\Delta(z,x) 
   & = &
   \sum_{n\ge 2} z^n \sum_{k=0}^{n-1} f_k^{(n)} x^k \notag\\
   & = &
   \frac{z^2}{1-z}\cdot
   \frac{1+(1+x)z(z+x-2)}{(xz+z-1)^2\big((1+x)z(z-2)+1\big)}
   \label{eq:lam}\\
   & = &
   z^2 + (3 + 3x + x^2) z^3 + 
   (6 + 13x + 12x^2 + 4x^3) z^4 + O(z^5) .\notag
\end{eqnarray}

Some remarks are in order here: First, we have removed a factor `$x$'
from the formula in \cite[Remark~5.4]{Lam-Postnikov06}, because
Lam and Postnikov's normalization of the $f$-vector is different from
ours. Secondly, in this subsection we do not include a count for the
empty face in our generating functions, because we want to combine
triangulations into larger ones, but prefer not to be bothered by the
fact that each triangulation has a \emph{unique} empty face; we will
remedy this starting from Section~\ref{sec:dncombine}. And finally, note
that the $z^2$~term in the formula corresponds to the unique
$0$-dimensional face of~$\Delta(2,2)$.

These remarks out of the way, we count the number of
$(k\ge0)$-dimensional faces of $\triang \P_{D_n,\total\Delta_{2,n}}$
by an inclusion-exclusion argument, making use of the following two
facts: For $k\ne1$, each $k$-dimensional face of a hypersimplex is
again a hypersimplex; and for $n\ge3$ each hypersimplex facet
of~$\P_{D_n}$ is adjacent to exactly $n$ other hypersimplex
facets. (There is an additional oddity for $n=3$, in that it remains
true that each of the 8~triangles~$\Delta(3,2)$ in $\P_{D_3}$ is
adjacent to exactly $3$~other ones, but the adjacency happens via
vertices, not codimension~1 faces.)

The count must start out $2^n f_k^{(n)}$, but then we have overcounted
the $k$-faces in the intersections of two hypersimplex facets. For
each such intersection, we must subtract $f_k^{(n-1)}$, so it remains
to calculate the number of such adjacencies of hypersimplex
facets. Without loss of generality, we can assume one of the
hypersimplex facets to lie in the hyperplane $\sum_{i=1}^n x_i =
\langle a,x\rangle = 2$, where $a=(1,\dots,1)$ is the all-1 vector.
This facet $F$ is adjacent to $n$~other hypersimplex facets
$F_1,\dots, F_n$, namely the ones defined by the normal vectors
obtained from $a$ by changing exactly one 1 to $-1$. A normal
vector that selects the $(n-2)$-dimensional hypersimplex $F\cap F_j$
is obtained from $a$ by setting the $j$-th entry to $0$.

For $i$ with $1\le i\le n-2$, any $(n-i)$-dimensional
hypersimplex that is the intersection of $i$ hypersimplex facets of
$\P_{D_n}$ lies in the hyperplane $\langle a',x\rangle = 2$, where
$a'$ is a vector with $i-1$ entries~`$0$' and the remaining $n-i+1$
entries either `$+1$' or `$-1$'. Since there are
$\binom{n}{i-1}2^{n-i+1}$ such vectors, we obtain
\begin{eqnarray}
    f_k(D_n,\total\Delta_{2,n}) 
    & = &
    \sum_{i=1}^{n-2} (-1)^{i-1} \binom{n}{i-1} 2^{n-i+1}
    f_k^{(n-i+1)}\label{eq:fksum1}\\
    & = &
    \sum_{j\ge0} (-1)^j \binom{n}{j} 2^{n-j} f_k^{(n-j)}.\notag
\end{eqnarray}
Note that the last sum only runs up to $n-k$ because  $f_k^{(l)}=0$
for $k>l$.

Let's discuss some special cases: The sum \eqref{eq:fksum1} is also
valid for $n=3$; for $i\ge n$ because $f^{(l)}_k=0$ for $l\le 1$ and
all $k$; and for $i=n-1$ (the case of edges) because $f^{(2)}_0=1$ is
the only nonzero value of $f^{(2)}$, so that
$f_0(D_n,\total\Delta_{2,n})$ is the only affected term; but the
number of vertices works out correctly.  This comes about because
$f_0(\Delta_{2,n})=\binom{n}{2}$, so that
\[
  f_0(D_n,\total\Delta_{2,n}) 
  \ = \
  \sum_{j\ge0} (-1)^j \binom{n}{j} 2^{n-2} \binom{n-j}{2}
  \ = \
  2n(n-1),
\]
the correct number of vertices.

We proceed to calculate the corresponding
generating functions:
\begin{eqnarray*}
    f_{D_n,\total\Delta_{2,n}} (x)
    &=& 
    \sum_{k=0}^{n-1} f_k(D_n,\total\Delta_{2,n}) x^k\\
    &=& 
    \sum_{i=0}^{n-1} (-1)^i \binom{n}{i} 2^{n-i} 
    \sum_{k=0}^{n-i} f_k^{(n-i)} x^k \\
    &=&
    \sum_{i=0}^{n-1} (-1)^i \binom{n}{i} 2^{n-i} 
    f^{(n-i)} (x) \\
    &=&
    (-1)^n \sum_{i=1}^{n} (-2)^i\binom{n}{i} 
    f^{(i)} (x) ,
\end{eqnarray*}
and 
\begin{eqnarray*}
  F_{D_n,\total\Delta_{2,n}}(z,x)
  &=&
  \sum_{n\ge 2} f_{D_n,\total\Delta_{2,n}}(x) z^n \\
  &=&
  \sum_{i\ge 1} (-2)^i f^{(i)}(x) 
  \sum_{n\ge i} (-z)^n \binom{n}{i} \\
  &=&
  \sum_{i\ge 1} \frac{(2z)^i}{(1+z)^{i+1}} f^{(i)} (x)\\ 
  &=&
  \frac{1}{1+z} 
  \sum_{i\ge1} \left(\frac{2z}{1+z}\right)^i f^{(i)}(x) \\
  &=& 
  \frac{1}{1+z} 
  F_\Delta\left( \frac{2z}{1+z}, x \right) \\
  &=&
  \frac{4z^2}{(z-1)^2-4xz}
  \left(\frac{2xz(x+1)}{(2xz+z-1)^2} - \frac{1}{z-1}\right),
\end{eqnarray*}
where $F_\Delta(z,x)=\sum_{i\ge0} f^{(i)}(x)z^i$ is the generating
function of the $f$-vectors of triangulations of the second
hypersimplex from~\eqref{eq:lam}; note that
$f^{(0)}(x)=f^{(1)}(x)=0$.

\subsection{The total generating function}\label{sec:dncombine}

It remains to combine the two previous generating functions:
\[
    F_{\triang\partial\P_{D_n}}(z,x) 
    \ = \ 
      F_{D_n,\total\interior\Diamond_{n-1}}(z,x) 
    + F_{D_n,\total\Delta_{2,n}}(z,x) 
    + \frac1x \sum_{n\ge2} z^n,
\]
where the last term corresponds to a count for the empty face in each
dimension. We could now extract the $f$-vector via a 
routine calculation;
however, it will be easier to do this with the
$h$-vector in hand. 

For this, we form a unimodular triangulation~$\Delta$ of the interior
of $\P_{D_n}$ by coning over $\triang\partial\P_{D_n}$ from the
origin. However, since we are interested only in the $h$-vector of the
resulting triangulation, we can use two well-known properties of
$h$-polynomials to simplify things.  Namely, the $h$-polynomial
$h_{\Delta_1\star\Delta_2}(x)=h_{\Delta_1\star\Delta_2}$ of the join
of two simplicial complexes equals the product $h_{\Delta_1}
h_{\Delta_2}$ of the individual $h$-polynomials; moreover, the
$h$-polynomial of a point is $h_{\{0\}}=1$. Thus,
\[
   h_{\Delta}
   \ = \
   h_{(\triang\partial\P_{D_n})\star\{0\}} 
   \ = \ 
   h_{\triang\partial\P_{D_n}}, 
\]
and it suffices to calculate the $h$-polynomial corresponding to
$\triang\partial\P_{D_n}$.
For this, we need to re-index each polynomial
\[
  [z^n]F_{\triang\partial\P_{D_n}}(z,x) 
  \ = \ 
  \sum_{i=0}^{n} f_{i-1}x^{i-1} 
\]
to read instead $\sum_{i=0}^{n} f_{i-1} x^{n-i}$; we achieve this by
setting
\[
   \widetilde F_{\triang\partial\P_{D_n}}(z,x) 
   \ = \ 
   \frac{1}{x} F_{\triang\partial\P_{D_n}}\bigg(zx,\frac1x\bigg). 
\]
Since $h_{\Delta}(x)=f_{\Delta}(x-1)$, we obtain the generating
function for the $h$-polynomials as
\begin{eqnarray*}
  H_{D}(z,x) &=& \widetilde F_{\triang\partial\P_{D_n}}(z,x-1)\\
  &=&
  z^2\frac{(1+x)^2 - 3(1+x)(1+x^2)z + (3+x^2)(1+3x^2)z^2 -
    (x-1)^2(1+x)(1+x^2)z^3}{(xz+z-1)^2(1-2z(1+x)+z^2(x-1)^2)}\\
  &=&
  \frac{z^2}{2}\left(
    \frac{1}{\frac{1}{(\sqrt{x}-1)^4} - \frac{z}{(\sqrt{x}-1)^2}} + 
    \frac{1}{\frac{1}{(\sqrt{x}+1)^4} - \frac{z}{(\sqrt{x}+1)^2}}
  \right)
  +\frac{2xz^2(xz+z-2)}{(xz+z-1)^2} \\
  &=& \phantom{+{}}
  \big(1 + 2 x + x^2 \big) z^2 \\ &&
  + \big(1 + 9 x + 9 x^2 + x^3 \big) z^3 \\ &&
  + \big(1 + 20 x + 54 x^2 + 20 x^3 + x^4 \big) z^4 \\ &&
  + \big(1 + 35 x + 180 x^2 + 180 x^3 + 35 x^4 + x^5 \big) z^5
  + O(z^6).
\end{eqnarray*}
It is now straightforward to check that $H_{D}(z,x)=\sum_{n\ge2} p_n(x)
z^n$ is the generating function for Colin Mallows's formula conjectured in
Conway and Sloane~\cite{Conway-Sloane91}:
\[
   p_n(x) \ = \
   \frac{(1+\sqrt{x})^{2n} + (1-\sqrt{x})^{2n}}{2} -
   2nx(1+x)^{n-2},
\]
and from this we extract the $f$-vector stated in Theorem 1 without
too much difficulty.

\subsection{The growth series of $C_n$: reprise}\label{sec:reprise}

We are now in a position to give a triangulation-theoretic derivation
of the growth series of $C_n$. For this, notice from the description
in Section~\ref{sec:Cn} that the lattice points in~$\P_{C_n}$ are the
vertices of $\P_{D_n}$, together with the vertices of~$2\Diamond_n$.
We can therefore build a unimodular triangulation of the boundary of
$\P_{C_n}$ by starting from our unimodular triangulation of
$\P_{D_n}$. However, instead of triangulating the $2n$~cross-polytope
facets, we cone over each such facet $F$ from the corresponding vertex
$v_F$ of $2\Diamond_n$.
%
%
By projecting $v_F$ into the barycenter of~$F$, we
see that counting the number of up to $(n-1)$-dimensional faces added
by each coning operation amounts to counting the total number of
interior faces in the triangulation of $\Diamond_{n-1}$ obtained by
coning over the origin. These numbers, in turn, are encoded in the
generating function
\[
   f_{\interior(\{0\}\star\partial\Diamond_{n-1})}(x) 
   \ = \
   x\left( f_{\partial\Diamond_{n-1}}(x) - \frac1x \right) + 1 
   \ = \
   (1+2x)^{n-1}.
\]
As before, the faces corresponding to the cones over different
cross-polytope facets are disjoint, so that their total number is
counted by the generating function
\[
  F_{C_n,\total\Diamond_{n-1}}(z,x) 
  \ = \
  2z \frac{\partial}{\partial z}
  \sum_{n\ge2} (1+2x)^{n-1} z^n 
  \ = \
  -\frac{2(2x+1) z^2 ( 2xz+z - 2)}{(2xz+z-1)^2}.
\]
Combining generating functions as before, we obtain
\begin{eqnarray*}
  (x-1)H_C (z,x)
  &=&
  \phantom{{}+{}}
  F_{C_n,\total\Diamond_{n-1}}\left(z(x-1),\frac{1}{x-1}\right) \\
  &&{}+ F_{D_n,\total\Delta_{n,2}} \left(z(x-1),\frac{1}{x-1}\right) \\ 
  &&{}+(x-1) \sum_{n\ge2} \left(z(x-1)\right)^n.
\end{eqnarray*}
\begin{eqnarray*}
  H_C(z,x)
  &=&
  z^2 \frac{1+6x+x^2 - (x-1)^2(x+1)z}%
  {1 - 2(x+1)z + (x-1)^2 z^2} \\
  &=&
  \phantom{+{}}
  \big(1 + 6 x + x^2 \big) z^2 \\ && 
  + \big(1 + 15 x + 15 x^2 + x^3 \big) z^3 \\ && 
  + \big(1 + 28 x + 70 x^2 + 28 x^3 + x^4\big) z^4 \\ && 
  + \big(1 + 45 x + 210 x^2 + 210 x^3 + 45 x^4 + x^5 \big) z^5
  + O(z^6).
\end{eqnarray*}
Again, it is easily checked that $H_C(z,x) = \sum_{n\ge2} p_n(x) z^n$
with
\[
   p_n(x) \ = \ \sum_{k=0}^n \binom{2n}{2k} x^k.
\]

\bibliographystyle{amsplain}
\bibliography{growth-series.final}

\end{document}